\documentclass[12pt]{amsart}
\usepackage{enumerate}
\usepackage{url}
\usepackage{xspace}

\usepackage[margin=1in]{geometry}
\DeclareMathOperator{\pnt}{\raise 0.5mm \hbox{\large\textbf{.}}}

\newcommand{\note}[2][ ]{}%dummy macro

\newtheorem{theorem}{Theorem}%[section]
\newtheorem{lemma}[theorem]{Lemma}

\theoremstyle{definition}

\newtheorem{example}[theorem]{Example}

\usepackage[pdfauthor={Huy Tai Ha and Fabrizio Zanello and Erik Stokes},
            pdftitle={Pure O-sequences and matroid h-vectors},
            pdfsubject={commutative algebra},
            pdfkeywords={simplicial complex, matroid, h-vector,
              pure O-sequence},
            pdfproducer={Latex with hyperref},
            pdfcreator={latex->dvips->ps2pdf},
            pdfpagemode=UseNone,
            bookmarksopen=false,
            bookmarksnumbered=true]{hyperref}

\title[Odd values of  the Klein $j$-function and the cubic partition function]{On the number of odd values of  the Klein $j$-function\\and the cubic partition function}
\author{Fabrizio Zanello} \address{Department of Mathematical  Sciences\\ Michigan Tech\\ Houghton, MI  49931-1295}
\email{zanello@mtu.edu,{\ }zanello@math.mit.edu}

\begin{document}
\maketitle
%\linenumbers

\begin{abstract} In this note, using entirely algebraic or elementary methods, we determine a new asymptotic lower bound for the number of odd values of one of the most important modular functions in number theory, the Klein $j$-function. Namely, we show that the number of integers $n\le x$ such that the Klein $j$-function --- or equivalently, the cubic partition function ---  is odd is at least of the order of
$$\frac{\sqrt{x}  \log \log x}{\log x},$$
for $x$ large. This improves recent results of Berndt-Yee-Zaharescu and Chen-Lin, and approaches significantly the best lower bound currently known for the ordinary partition function, obtained using the theory of modular forms. Unlike many works in this area, our techniques to show the above result, that have in part been inspired by some recent ideas of P. Monsky on quadratic representations, do not involve the use of modular forms.

Then, in the second part of the article, we  show how to employ modular forms in order to  slightly refine our bound. In fact, our brief argument, which combines a  recent result of J.-L. Nicolas and J.-P. Serre with a classical theorem of J.-P. Serre on the asymptotics of the Fourier coefficients of certain level 1 modular forms, will more generally apply to provide a lower bound for the number of odd values of any positive power of the generating function of the partition function.

\end{abstract}

\section{Introduction and past work}

One of the most interesting, and still poorly understood, problems in partition theory is the study of the parity of the partition function $p(n)$. It is widely believed (see e.g. \cite{PS}) that $p(n)$ is both even and odd  with density $1/2$; i.e., if $f_p^0(x)$ (respectively, $f_p^1(x)$) denotes the number of  integers $n\le x$ such that $p(n)$ is even (respectively, odd), then  $f_p^0(x)$ and $f_p^1(x)$ are conjectured to be both asymptotic to $x/2$, for $x$ large. However, despite important efforts by many researchers using a variety of combinatorial, algebraic or analytic tools, this conjecture still appears to be  out of reach for today's mathematics. The current best lower bounds  for $f_p^0(x)$ and $f_p^1(x)$ are due to J. Bella\"iche and J.-L. Nicolas \cite{BN}, who  refined results by S. Ahlgren, D. Eichhorn, K. Ono, and J.-P. Serre, among many others (see, as a sample, \cite{Ahl,Ei,Nic,Ono,Ser}). Namely, using the theory of modular forms, Bella\"iche and Nicolas proved:
$$f_p^0(x) \gg \sqrt{x} \log \log x, {\ }{\ }\text{and}{\ }{\ }f_p^1(x)\gg \frac{\sqrt{x}}{(\log x)^{7/8}}.$$

Two closely related functions, whose parity seems equally hard to understand, are the  \emph{Klein $j$-function}, $j(n)$, and the \emph{cubic partition function}, $c(n)$. The Klein $j$-function, a modular function playing a central role in several areas of number theory (see e.g. \cite{Ap}), is defined by:
\begin{equation}\label{tau}
\sum_{n\ge -1}j(n)q^n =\frac{\left(1+240\sum_{n\ge 1}\frac{n^3q^n}{1-q^n}\right)^3}{q\prod_{i\ge 1} (1-q^i)^{24}}.
\end{equation}

The cubic partition function denotes the number of partitions of $n$ into parts of two kinds, where those  of the second kind can only be even. Thus,  the generating function for $c(n)$ is:
$$\sum_{n\ge 0}c(n)q^n = \frac{1}{\prod_{i\ge 1} (1-q^i)(1-q^{2i})}.$$

W.Y.C. Chen and B.L.S. Lin \cite{CL} and, in different but equivalent terms, K. Ono and N. Ramsey \cite{OR} have conjectured that, like for $p(n)$, a ``$50\%$ density'' result also holds for $c(n)$. 

From an algebraic perspective, it is easy to see that, in the ring of series $\mathbb Z_2[[q]]$, the generating function for $c(n)$ is the multiplicative inverse of the series of the triangular numbers $\binom{n+1}{2}$, and that  it  coincides with the cube of the generating function for $p(n)$. The generating function for $j(n)$ is essentially the inverse of the series  of the odd squares (shifted by one degree), or of  the Ramanujan Tau function. Similarly to $p(n)$, determining optimal bounds on the parity of $j(n)$ and $c(n)$ is considered a nearly intractable problem today. We just remark here that the difficulty of this parity problem appears to be shared by many  series that,  modulo 2, are the  inverse of a quadratic series in $\mathbb Z_2[[q]]$ --- including indeed $p(n)$, because of the Pentagonal Number Theorem (see \cite{CEO,Mo2} for more). 

However, a large amount of work has recently been done on the parity of $j(n)$ and $c(n)$, by a variety of  methods; for some important contributions, see e.g.: C. Alfes \cite{Alf}, who proved, among other facts, a lower bound of the order of $\sqrt{x}/\log x$ for the even values of $j(n)$ when $n\equiv 7$ (mod 8), and thus for the even values of $c(n)$ (as we will see below, these two bounds are essentially equivalent); W.Y.C. Chen and B.L.S. Lin \cite{CL}, who proved the infinity of both the even and the odd values; and K. Ono and N. Ramsey \cite{OR}, who refined some of Alfes's results. The  best existing bounds, both for the even and the odd values, are due to B. Berndt, A.J. Yee and A. Zaharescu \cite{BYZ}. They proved that $f_j^0(x)$ (for $n\equiv 7$ (mod 8)) and $ f_c^0(x) \gg \sqrt{x}$, while $f_j^1(x)$ and $f_c^1(x)  \gg x^{1/2 -\alpha/\log \log x}$, for some constant $\alpha$.

In this brief note, our main goal is refine the above bound for the number of odd values of $j(n)$ and $c(n)$, using only algebraic or elementary methods. We will show that
$$f_j^1(x) {\ }{\ }\text{and}{\ }{\ } f_c^1(x)  \gg \frac{\sqrt{x} \log \log x}{\log x},$$
which is just a factor of $(\log x)^{1/8}/\log \log x$ away from the current best bound for $p(n)$ \cite{BN}. 

Unlike many works in this area, our approach will completely avoid the complex-analytic machinery of  modular forms. Interestingly, we will also relate this problem to the study of integer representations  by quadratic forms;  though it may not be immediately obvious from  our proof, this work was in part   inspired by a recent comment of P. Monsky on \emph{MathOverflow} in response to a question by J. Bella\"iche \cite{Mo}, as well as by paper \cite{Mo2}, again by Monsky.

In the final portion of the paper, we will then see how to  slightly improve the above bound with another brief argument that employs the theory of modular forms, by combining a classical result of J.-P. Serre \cite{Ser2} and a  recent one by J.-L. Nicolas and J.-P. Serre \cite{NS} on the asymptotics of the number of odd Fourier coefficients of those modular forms of level 1 that can be expressed as a pure power of the modulo 2 reduction of the Ramanujan Tau function. In fact, more generally, we will show that the same refined lower bound holds for any positive power $b$ of the generating function of the partition function (the cubic partitions and the Klein $j$-function correspond to the case $b=3$).

\section{The lower bound}

We begin with a key lemma on integer representations as sums of four squares.

\begin{lemma}\label{78}
Let $R(n)$ be the number of representations of $n$ as $n=X^2+Y^2+Z^2+4W^2$, for odd positive integers $X,Y,Z$ and $W$, and let $f^1_R(x)$ denote the number of  integers $n\le x$ such that $R(n)$ is odd. Then
$$f^1_R(x)\gg \frac{x \log \log x}{\log x}.$$
\end{lemma}

\begin{proof}
Notice that we can assume in the proof that $n\ge 7$, $n\equiv 7$ (mod 8), otherwise $R(n)=0$. Therefore, if we can express $n$ as $n=X^2+Y^2+Z^2+T^2$, for \emph{arbitrary} integers $X,Y,Z$ and $T$, then exactly three of  $X,Y,Z$ and $T$ must be odd and one must be congruent to 2 modulo 4; say $T=2W$, with $W$ odd. In particular, none of them is zero, and the positive and the  negative solutions are symmetric for each variable. It easily follows that, if $M(n)$ is the number of arbitrary integer representations of $n=X^2+Y^2+Z^2+T^2$, then
$$R(n)=M(n)/(2^4 \cdot 4)=M(n)/64.$$

By the four-square theorem, we know that $M(n)=8\sum d$, where the sum is taken over the divisors $d$ of $n$ that are not divisible by 4. Since $n\equiv 7$ (mod 8) is odd, this immediately gives us that
\begin{equation}\label{R}
R(n)=8\sigma(n)/64=\sigma(n)/8,
\end{equation}
where as usual  $\sigma$ denotes the sum-of-divisor function. It is folklore that $\sigma$ is multiplicative and that, if we write the prime factorization of $n$ as $n=\prod_{i=1}^r p_i^{a_i}$, then
\begin{equation}\label{sigma}
\sigma(n)=\prod_{i=1}^r (1+p_i+p_i^2+\dots +p_i^{a_i}).
\end{equation}

Notice that, by (\ref{R}), $R(n)$ is odd if and only if $\sigma(n)$ is not divisible by 16. In other words,  showing the lemma is tantamount to proving that, for at least the order of $x\log \log x /\log x$ positive integers $n\le x$, $n\equiv 7$ (mod 8), we have the 2-adic valuation $v_2(\sigma(n))=3$. But if $n$ is the product of exactly two prime factors, say $n=pq$, where $p\equiv 3$ (mod 8) and $q\equiv 5$ (mod 8), then by (\ref{sigma}),
$$v_2(\sigma(n))=v_2((1+p)(1+q))=2+1=3.$$

The result now follows by observing that, by an elementary generalization of the prime number theorem for arithmetic progressions to 2-almost primes (see e.g. \cite{Lu}), the number of the above integers $n=pq$ is asymptotic to a constant multiple of $x\log \log x /\log x$.
\end{proof}

Our main result of this section is the following asymptotic lower bound for $f_c^1(x)$ and $f_j^1(x)$. As usual, we say that two series $\sum_{n}a(n)q^n \equiv \sum_{n}b(n)q^n$ (mod $m$) if $a(n)\equiv b(n)$ (mod $m$), for all $n$.

\begin{theorem}\label{main}
We have:
$$f_j^1(x) {\ }{\ }\text{and}{\ }{\ } f_c^1(x)  \gg \frac{\sqrt{x} \log \log x}{\log x}.$$
\end{theorem}

\begin{proof} We first notice that it suffices to show the result for $c(n)$. Indeed, working modulo 2 (here and in all congruences throughout this proof), by (\ref{tau}) the generating function for $j(n)$ is the multiplicative inverse of the Ramanujan Tau function; i.e.,
$$\sum_{n\ge -1}j(n)q^n \equiv \frac{1}{q\prod_{i\ge 1} (1-q^i)^{24}}.$$

Hence,
$$\sum_{n\ge -1}j(n)q^n \equiv \frac{1}{q\prod_{i\ge 1} (1-q^{8i})^3}\equiv \frac{1}{q\prod_{i\ge 1} (1-q^{8i})(1-q^{8\cdot 2i})}= q^{-1}\sum_{n\ge 0}c(n)q^{8n}.$$

Therefore, it is clear that a lower bound for $f_c^1(x)$ is equivalent, up to a multiplicative constant, to the same lower bound for $f_j^1(x)$, as we claimed. 

Thus, let us now prove the theorem for $f_c^1(x)$. Recall that, in Lemma \ref{78}, we have shown the existence of $\gg x\log \log x /\log x$  integers $n\le x$  having an odd number of representations, $R(n)$, of the form $n=X^2+Y^2+Z^2+4W^2$, for odd positive  integers $X,Y,Z$ and $W$.

Translating this in terms of generating functions, one moment's thought gives us that, if we define $\Delta(q)= \sum_{n\ge 0}q^{(2n+1)^2}$ to be the generating function of the odd squares, then
\begin{equation}\label{8m}
\sum_{n\ge 0}R(n)q^{n} = \Delta^3(q) \cdot \Delta(q^4).
\end{equation}

But by the following well-known modulo 2 identity (it  follows, for instance, from the Jacobi triple product identity),
\begin{equation}\label{3}
\prod_{i\ge 1} (1-q^i)^3 \equiv \sum_{n\ge 0}q^{\binom{n+1}{2}},
\end{equation}
we easily have: 
\begin{equation}\label{24}
\Delta(q) = \sum_{n\ge 0}q^{4n(n+1)+1} = q \sum_{n\ge 0}q^{8\binom{n+1}{2}} \equiv q\prod_{i\ge 1} (1-q^{8i})^{3}\equiv q\prod_{i\ge 1} (1-q^{i})^{24}.
\end{equation}
(This fact was already known to S. Ramanujan; see \cite{AB}, page 84. We thank the referee for suggesting this reference.) Therefore, (\ref{8m}) yields:
$$\sum_{n\ge 0}R(n)q^{n} \equiv \left(q\prod_{i\ge 1} (1-q^{i})^{24}\right)^ 3 \cdot q^4\prod_{i\ge 1} (1-q^{4i})^{24}$$$$\equiv q^7\prod_{i\ge 1} (1-q^{i})^{72} \cdot \prod_{i\ge 1} (1-q^{i})^{96}=q^7\prod_{i\ge 1} (1-q^{i})^{168}\equiv q^7\prod_{i\ge 1} (1-q^{8i})^{21}.$$

Thus, if we set $\prod_{i\ge 1} (1-q^{i})^{21}=\sum_{n\ge 0}b(n)q^{n}$, we have  $b(n)\equiv R(8n+7)$ for all $n$. Hence,  it promptly follows from Lemma \ref{78} that there are $\gg x\log \log x /\log x$  odd coefficients of $q\prod_{i\ge 1} (1-q^{i})^{21}$ in degree $\le x$.

Now notice that, by  (\ref{24}),
$$q\prod_{i\ge 1} (1-q^{i})^{21}=\frac{q\prod_{i\ge 1} (1-q^{i})^{24}}{\prod_{i\ge 1} (1-q^{i})^{3}}
\equiv \frac{\Delta(q)}{\prod_{i\ge 1} (1-q^i)(1-q^{2i})}= \Delta(q) \cdot \sum_{n\ge 0}c(n)q^{n}.$$

Therefore, since there are clearly the order of $ \sqrt{x}$ odd coefficients of $\Delta(q)$ in degree $\le x$, we easily deduce that
$$f_c^1(x)\gg \frac{x \log \log x}{\log x}\cdot \frac{1}{\sqrt{x}} = \frac{\sqrt{x} \log \log x}{\log x},$$
and the proof  is complete.
\end{proof}

\section{The modular form refinement}

As we have mentioned in the introduction, our goal in this section is to slightly refine our lower bound of Theorem \ref{main} using the theory of modular forms. We will briefly recap here the necessary terminology and results, referring e.g. to \cite{Ap,DS} for an introduction to modular forms, and to \cite{Ser1} for more on those of level 1 modulo 2, which are the topic of this section.

Let $M$ be the algebra of all modular forms of level 1 with integer coefficients. A nice result of P. Swinnerton-Dyer says that, modulo 2, $M=\mathbb Z_2[\Delta]$. In other words, the reduction modulo 2 of any modular form of level 1 can be expressed as a polynomial in $\Delta(q)= \sum_{n\ge 0}q^{(2n+1)^2}$. A classical result of J.-P. Serre  \cite{Ser2} then gives, asymptotically, the order of magnitude of the number of odd Fourier coefficients of any modular form $m$ of level 1 having \emph{order of nilpotency} $g(m)\ge 2$; in particular, all positive powers $\Delta^k$ of $\Delta$, where $k$ is not a power of 2, can be proven to satisfy $g(\Delta^k)\ge 2$. Serre's result can be phrased as follows:

\begin{theorem}[Serre \cite{Ser2}] Let $m$ be a  modular form of level 1 with integer coefficients such that $g(m)\ge 2$. With a slight abuse of notation, denote by $f^1_m(x)$ the number of odd Fourier coefficients of  $m$ up to degree $x$. Then the order of magnitude of $f^1_m(x)$ is asymptotic to a constant multiple of
$$\frac{x (\log \log x)^{g(m)-2}}{\log x}.$$
\end{theorem}

Notice  that, indeed, Serre's theorem cannot apply to $m=\Delta^k$ for $k=2^t$, because obviously, working modulo 2, $f^1_{\Delta^{2^t}}(x)$ has the order of magnitude of $f^1_{\Delta}(x)$, namely $\sqrt{x}$. We also remark that an asymptotic equivalent of $f^1_m(x)$ has been given in \cite{BN}, as pointed out to us by the referee.

For any integer $k\ge 1$, write the binary expansion of $k$ as $k=\sum_{i\ge 0}\beta_i 2^i$ (hence $\beta_i\in \{0,1\}$ for all $i$).  Define now the \emph{height} of $k$ as $h(k)=n_3(k)+n_5(k)$, where
$$n_3(k)=\sum_{i\ge 0}\beta_{2i+1} 2^i {\ }{\ }\text{and}{\ }{\ } n_5(k)=\sum_{i\ge 0}\beta_{2i+2} 2^i.$$

A recent crucial result of  J.-L. Nicolas and J.-P. Serre  \cite{NS} says (among other facts) that, for any odd integer $k\ge 3$, 
\begin{equation}\label{oddk}
g(\Delta^k)=h(k)+1.
\end{equation}

Therefore, by Serre's theorem and Identity (\ref{oddk}), we can easily deduce that, for any $k\ge 3$ odd, the number of odd coefficients  of $\Delta^k$ up to degree $x$ is, asymptotically, of the order of magnitude of 
\begin{equation}\label{kk}
\frac{x (\log \log x)^{n_3(k)+n_5(k)-1}}{\log x}.
\end{equation}

\begin{example}
Let $k=7$. We have $7=1\cdot 2^0+1\cdot 2^1+1\cdot 2^2$, and therefore $n_3(7)=n_5(7)=1$ and $n_3(7)+n_5(7)-1=1$. Thus, by (\ref{kk}), the order of magnitude of $f^1_{\Delta^7}(x)$, or equivalently,  of the number of odd coefficients of $\prod_{i\ge 1} (1-q^{i})^{21}$ up to degree $x$, is  $x\log \log x /\log x$. This also proves that our lower bound of Lemma \ref{78} is optimal. 
\end{example}

The next theorem is the main result of this section, where we employ the modular form machinery developed in  \cite{NS,Ser2} to refine the elementary lower bound given in the previous section for the number of odd coefficients of the Klein $j$-function and the cubic partition function, by a factor of $(\log \log x)^K$, for any $K> 0$. In fact, our result will  more generally hold for any positive power $b$ of the generating function of the partition function, that we define as
$$\sum_{n\ge 0}P_b(n)q^n = \frac{1}{\prod_{i\ge 1} (1-q^i)^b}.$$
(Notice that $P_1(n)=p(n)$, and modulo 2, $P_3(n)=c(n)$.) 

We only remark here that, though it will not be part of this note, by a much longer and more technical analysis involving a similar modular form approach, and entirely along the lines of Theorem 6 of a recent paper of J. Bella\"iche and J.-L. Nicolas \cite{BN}, one could further slightly improve the lower bound of Theorem \ref{bbb} below. In fact, also the  lower bound of $\sqrt{x}$ for  the even coefficients obtained by Berndt-Yee-Zaharescu \cite{BYZ} could be slightly refined (again for all powers of the generating function of the partition function) exactly with the same type of argument given in \cite{BN}, Theorem 5. However, unfortunately, we have not yet been able to reach, by any approach (either involving or avoiding modular forms), a lower bound of at least the order of $\sqrt{x}$ for the number of odd coefficients of the generating function of the partition function or of any of its powers.

\begin{theorem}\label{bbb} For any given integer $b\ge 1$, the following lower bound holds, for any $K> 0$:
$$f_{P_b}^1(x) \gg \frac{\sqrt{x} (\log \log x)^K}{\log x}.$$
\end{theorem}

\begin{proof} First notice that it is a simple exercise on binary expansions to show that, if we set $\alpha(k)=n_3(k)+n_5(k)-1$, then $\lim_{k \rightarrow \infty}\alpha(k) =\infty$. Also, we claim that we may assume in the proof that $b$ is odd. Indeed, working modulo 2 (here and in all of the following congruences), if $b=2^s\cdot b_0$ for some odd integer $b_0$, then
$$\frac{1}{\prod_{i\ge 1} (1-q^i)^b}\equiv \frac{1}{\prod_{i\ge 1} (1-q^{2^s\cdot i})^{b_0}}.$$
Therefore, $f_{P_b}^1(x)$ is  (always up to a constant factor) of the same order of magnitude of $f_{P_{b_0}}^1(x)$, as desired.

We begin with the case  $b\equiv \pm 1$ (mod 3). Fix $K> 0$. Since $\alpha(k)$ goes to infinity for $k$ large, it is easy to see that there exists an integer $k$ such that $\alpha(k)\ge K$ and $3k=2^t-b$, for some $t$. (The integer $t$ will be even if $b\equiv 1$ (mod 3), and odd if $b\equiv -1$ (mod 3).) 

Therefore, by congruence (\ref{24}), we get:
$$\Delta^k(q)= \left( \sum_{n\ge 0}q^{(2n+1)^2}\right)^k \equiv \left( q\prod_{i\ge 1} (1-q^i)^{24}\right)^k=q^k\prod_{i\ge 1} (1-q^i)^{24k}$$$$\equiv q^k \prod_{i\ge 1} (1-q^{8i})^{3k}\equiv \frac{q^k \prod_{i\ge 1} (1-q^{8i})^{2^t}}{\prod_{i\ge 1} (1-q^{8i})^{b}}\equiv \frac{q^k \prod_{i\ge 1} (1-q^{2^{t+3}\cdot i})}{\prod_{i\ge 1} (1-q^{8i})^{b}}.$$

Notice that,  by (\ref{kk}),  $f^1_{\Delta^k}\gg x (\log \log x)^K/\log x$. Thus, since by the Pentagonal Number Theorem  there are  the order of $ \sqrt{x}$ odd coefficients of  $q^k \prod_{i\ge 1} (1-q^{2^{t+3}\cdot i})$ in degree $\le x$, we  easily conclude that 
$$f^1_{P_b}\gg \frac{x (\log \log x)^K}{\log x}\cdot \frac{1}{\sqrt{x}}=\frac{\sqrt{x} (\log \log x)^K}{\log x},$$
as desired.

The case $b\equiv 0$ (mod 3) can be treated in a rather analogous fashion, though now, fixed $K$, we want to choose our integer $k$ such that $\alpha(k)\ge K$ and $3k=3\cdot 2^t-b$. Then, by (\ref{3}) and (\ref{24}), we have (skipping some of the steps similar to the previous case):
$$\Delta^k(q) \equiv q^k \prod_{i\ge 1} (1-q^{8i})^{3k}= \frac{q^k \prod_{i\ge 1} (1-q^{8i})^{3\cdot 2^t}}{\prod_{i\ge 1} (1-q^{8i})^{b}}$$$$\equiv \frac{q^k \prod_{i\ge 1} (1-q^{2^{t+3}\cdot i})^3}{\prod_{i\ge 1} (1-q^{8i})^{b}}\equiv \frac{q^k \sum_{n\ge 0}q^{2^{t+3}\cdot \binom{n+1}{2}}}{\prod_{i\ge 1} (1-q^{8i})^{b}}.$$

Since,  by (\ref{kk}), $f^1_{\Delta^k}\gg x (\log \log x)^K/\log x$, and there are of course the order of $ \sqrt{x}$ odd coefficients of  $q^k \sum_{n\ge 0}q^{2^{t+3}\cdot \binom{n+1}{2}}$ in degree $\le x$, we can again conclude that 
$$f^1_{P_b}\gg \frac{\sqrt{x} (\log \log x)^K}{\log x},$$
and the proof of the theorem is complete. 
\end{proof}

\section{Acknowledgements} We thank the anonymous referee for helpful suggestions, and William Keith for comments and for several insightful conversations of partition theory. This work was done while the author was partially supported by a Simons Foundation grant (\#274577).


\begin{thebibliography}{llll}

\bibitem{Ahl} S. Ahlgren: \emph{Distribution of parity of the partition function in arithmetic progressions}, Indag. Math. (N.S.) \textbf{10} (1999), 173--181.

\bibitem{Alf} C. Alfes: \emph{Parity of the coefficients of Klein's $j$-function}, Proc. Amer. Math. Soc.  \textbf{141} (2013), no. 1, 123--130.

\bibitem{AB} G.E. Andrews and B. Berndt: ``Ramanujan's Lost Notebook, Part III'', Springer, New York (2012).

\bibitem{Ap} T.M. Apostol: ``Modular functions and Dirichlet series in number theory'', Graduate Texts in Mathematics \textbf{41}, Springer-Verlag, New York-Heidelberg (1976).

\bibitem{BN} J. Bella\"iche and J.-L. Nicolas: \emph{Parit\'e des coefficients de formes modulaires}, Ramanujan J., to appear. Available \href{http://people.brandeis.edu/~jbellaic/preprint/ParitecoefV5.pdf}{here}.

\bibitem{BYZ} B. Berndt, A.J. Yee and A. Zaharescu:  \emph{On the parity of  partition functions}, Internat. J. Math. \textbf{14} (2003), no. 4, 437--459.

\bibitem{CL} W.Y.C. Chen and B.L.S. Lin: \emph{Congruences for the Number of Cubic Partitions
Derived from Modular Forms}, preprint. Available on the  \href{http://arxiv.org/pdf/0910.1263.pdf}{arXiv}.

\bibitem{CEO} J.N. Cooper, D. Eichhorn and K. O'Bryant: \emph{Reciprocals of binary series}, Int. J. Number Theory \textbf{2} (2006), no. 4, 499--522. 

\bibitem{DS} F. Diamond and J. Shurman: ``A First Course in Modular Forms'', Graduate Texts in Mathematics \textbf{228}, Springer-Verlag, New York (2005).

\bibitem{Ei} D. Eichhorn: \emph {A new lower bound on the number of odd values of the ordinary partition function}, Ann. Comb. \textbf{13} (2009), 297--303.

\bibitem{Lu} \emph{MathOverflow} answer by ``Lucia'', February 10, 2014. Available \href{http://mathoverflow.net/questions/156982/chebotarev-density-theorem-for-k-almost-primes}{here}.

\bibitem{Mo} \emph{MathOverflow} answer by ``paul Monsky'', October 10, 2012. Available \href{http://mathoverflow.net/questions/100701/primes-and-x22y24z2/109275#109275}{here}.

\bibitem{Mo2} P. Monsky: \emph{Disquisitiones Arithmetic\ae {\ }and online sequence A108345}, preprint. Available on the  \href{http://arxiv.org/pdf/1009.3985.pdf}{arXiv}.

\bibitem{Nic} J.-L. Nicolas: \emph{Parit\'e des valeurs prises par la fonction de partition $p(n)$ et anatomie des entiers}, CRM Proceedings and Lecture Notes, Amer. Math. Soc. \textbf{46} (2008), 97--113.

\bibitem{NS} J.-L. Nicolas and J.-P. Serre: \emph{Formes modulaires modulo 2: l'ordre de nilpotence des op\'erateurs de Hecke}, C.R. Acad. Sci. Paris, Ser. I \textbf{350} (2012), 343--348.

\bibitem{Ono} K. Ono: \emph{Parity of the partition function}, Adv. Math. \textbf{225} (2010), no. 1, 349--366.

\bibitem{OR} K. Ono and N. Ramsey: \emph{A mod $\ell$ Atkin-Lehner theorem and applications}, Arch. Math. (Basel) \textbf{98} (2012), no. 1, 25--36.

\bibitem{PS} T.R. Parkins and D. Shanks: \emph{On the distribution of parity in the partition function}, Math. Comp. \textbf{21} (1967), 466--480.

\bibitem{Ser1} J.-P. Serre: \emph{Valeurs propres des op\'erateurs de Hecke modulo $\ell $}, Ast\'erisque \textbf{24-25} (1975), 109--117.

\bibitem{Ser2} J.-P. Serre: \emph{Divisibilit\'e de certaines fonctions arithm\'etiques}, L'Enseignement Math. \textbf{22} (1976), 227--260.

\bibitem{Ser} J.-P. Serre: Appendix to: J.-L. Nicolas, I.Z. Ruzsa and A. S\'ark\"ozy: \emph{On the parity of additive representation functions}, J. Number Theory \textbf{73} (1998), no. 2, 292--317.

\end{thebibliography}
\end{document}